\RequirePackage{silence}
\documentclass{article}

\usepackage{amsthm}

\newtheorem{theorem}{Theorem}

\newtheorem{lemma}{Lemma}
\newtheorem{definition}{Definition}

\newtheorem{example}{Example}

\newtheorem{remark}{Remark}

\usepackage{amsmath}
\usepackage{amssymb}
\usepackage{empheq}
\usepackage{xcolor, color, colortbl}
\usepackage{pifont}
\usepackage{enumitem}
\usepackage{framed}
\usepackage{graphicx} 

\usepackage{nicefrac}
\usepackage{booktabs}
\usepackage{multirow}
\usepackage{rotating}
\usepackage{nccmath}
\usepackage{wrapfig}
\usepackage{caption}

\usepackage[tikz]{bclogo}
\usepackage{wasysym}

\newenvironment{myenv}[1]
  {\mdfsetup{
    frametitle={\colorbox{white}{\space#1\space}},
    innertopmargin=0pt,
    frametitleaboveskip=-\ht\strutbox,
    frametitlealignment=\center
    }
  \begin{mdframed}%
  }
  {\end{mdframed}}

\definecolor{mydarkblue}{rgb}{0,0.08,0.45}
\usepackage[colorlinks=true,
    linkcolor=mydarkblue,
    citecolor=mydarkblue,
    filecolor=mydarkblue,
    urlcolor=mydarkblue,
    pdfview=FitH]{hyperref}

\definecolor{colormomentum}{HTML}{1b9e77}
\definecolor{colorstepsize}{HTML}{d95f02}
\definecolor{color3}{HTML}{7570b3}

\def\xx{{\boldsymbol x}}

\def\bb{{\boldsymbol b}}

\def\AA{{\boldsymbol A}}

\def\HH{{\boldsymbol H}}

\def\dif{\mathop{}\!\mathrm{d}}

\def\MP{\dif\mu_{\mathrm{MP}}}

\def\RR{{\mathbb R}}
\def\EE{{\mathbb{E}\,}}

\def\defas{\stackrel{\text{def}}{=}}

\newcommand*\mybluebox[1]{\colorbox{myblue}{\hspace{1em}#1\hspace{1em}}}

\definecolor{myblue}{HTML}{D2E4FC}
\definecolor{Gray}{gray}{0.92}

\usepackage{thmtools}
\usepackage{thm-restate}

\usepackage{hyperref}

\usepackage[accepted]{icml2020}

\icmltitlerunning{Universal Average-Case Optimality of Polyak Momentum}

\usepackage{fancyhdr} 
\begin{document}

\twocolumn[

    
    %
    \icmltitle{Universal Average-Case Optimality of Polyak Momentum}

    \icmlsetsymbol{equal}{*}
    
     \begin{icmlauthorlist}
        \icmlauthor{Damien Scieur}{equal,sail}
        \icmlauthor{Fabian Pedregosa}{equal,google}
    \end{icmlauthorlist}

    \icmlaffiliation{sail}{Samsung SAIT AI Lab, Montreal}
    \icmlaffiliation{google}{Google Research}

    \icmlcorrespondingauthor{Damien Scieur}{damien.scieur@gmail.com}

    \icmlkeywords{optimization, momentum, Polyak, acceleration, asymptotic}

    \vskip 0.3in
]
\printAffiliationsAndNotice{\icmlEqualContribution}

\begin{abstract}
Polyak momentum (PM), also known as the heavy-ball method, is a widely used optimization method that enjoys an asymptotic optimal worst-case complexity on quadratic objectives. However, its remarkable empirical success is not fully explained by this optimality, as the worst-case analysis --contrary to the average-case-- is not representative of the expected complexity of an algorithm. In this work we establish a novel link between PM and the average-case analysis. Our main contribution is to prove that \emph{any} optimal average-case method converges in the number of iterations to PM, under mild assumptions. This brings a new perspective on this classical method, showing that PM is asymptotically both worst-case and average-case optimal.
\end{abstract}

\section{Introduction}

Polyak momentum (PM), also known as the heavy-ball method, is a widely used optimization method. Originally developed to solve linear equations~\citep{frankel1950convergence, rutishauser1959theory}, it was generalized to smooth functions and popularized in the optimization community by Boris Polyak~\citep{polyak1964some,polyak1987introduction}.
This method has seen a renewed interest in recent years, as its stochastic variant which replaces the gradient with a stochastic estimate is effective on deep learning problems~\citep{sutskever2013importance, zhang20202dive}.

PM also enjoys a locally optimal rate of convergence for strongly convex and twice differentiable objectives. As is common within the optimization literature, this optimality is relative to the \emph{worst-case} analysis, that provides complexity bounds for \emph{any} input from a function class, no matter how unlikely.
Despite its widespread use, the worst-case is not representative of the typical behavior of optimization methods. The simplex method, for example, has a worst-case exponential complexity, that becomes polynomial when considering the average-case \citep{spielman2004smoothed}.

A more representative analysis of the typical behavior is given by the \emph{average-case} complexity, which averages the algorithm's complexity over all possible inputs. 
The average-case analysis is standard for analyzing sorting~\citep{knuth1997art} and cryptography~\citep{katz2014introduction} algorithms, to name a few.
However, little is known of the average-complexity of optimization algorithms, whose analysis depends on the often unknown probability distribution over the inputs. 

The recent work of \citet{pedregosa2020acceleration, lacotte2020optimal}
overcame this dependency on the input probability distribution through the use of random matrix theory techniques.
In the same papers, the authors noticed the convergence of some optimal average-case methods to PM, as the number of iterations grows (see Figure~\ref{fig:convergence_pm}). This is rather surprising given their crucial differences. For instance, average-case optimal methods use knowledge of the full spectral distribution, while PM only requires knowledge of its edges (i.e., smallest and largest eigenvalue).
Since this convergence was only shown on specific methods, it raises the question on whether this is a spurious phenomenon or if this holds more generally:

\begin{myenv}{Conjecture}
\centering As the number of iterations grows, all average-case optimal methods converge to Polyak momentum.

\vspace{-0.8em}\vphantom{1}
\end{myenv}

\begin{figure*}
    \centering
    \includegraphics[width=1\linewidth]{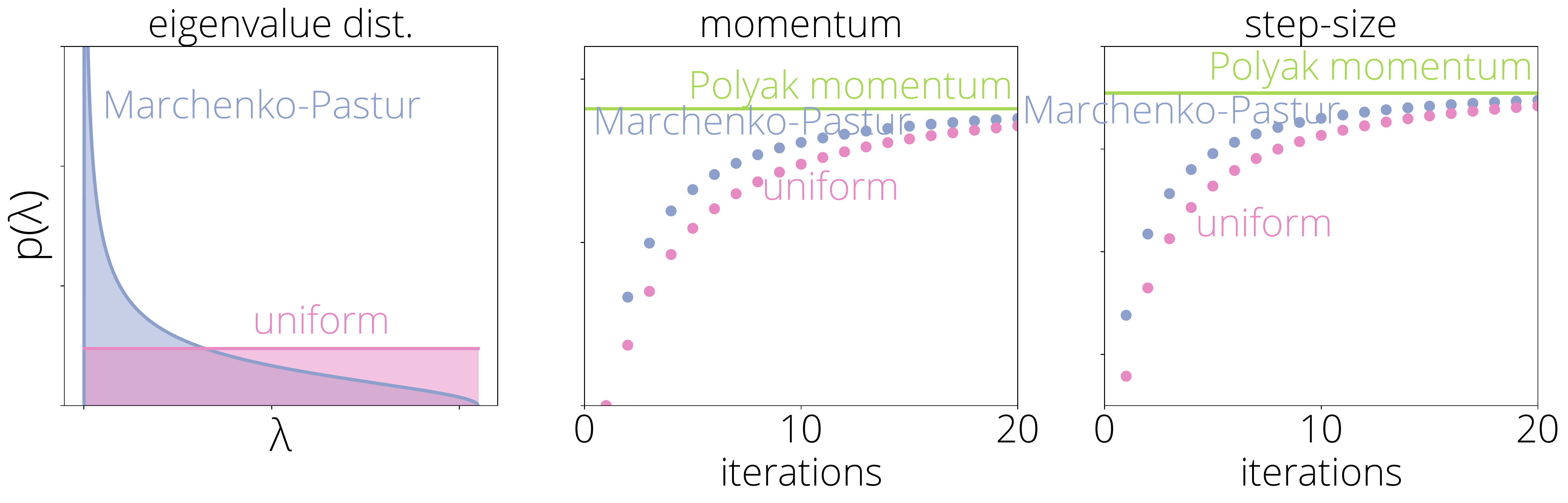}
    \caption{{\bfseries Convergence of optimal average-case methods to Polyak Momentum}. For the Marchenko-Pastur and uniform distribution of eigenvalues (left), we construct the method that has optimal average-case complexity and plot the momentum (middle) and steps-size (right) parameters. For the two methods considered, the momentum and step-size parameters converge as the number of iterations grows to those of Polyak momentum, displayed here as a straight line.
    }
    \label{fig:convergence_pm}
\end{figure*}

The {\bfseries main contribution} of this paper is to give a positive answer to this conjecture.
The main, but not so restrictive assumption, is that the probability density function of the eigenvalues is non-zero on the interval containing its support.
With this we can show the previously unknown property that PM is \emph{asymptotically} optimal under the average-case analysis, bringing a new perspectiveon the remarkable empirical performance of this classical method.
Furthermore, this statement is \emph{universal}, i.e., independent of the probability distribution over the inputs. 

\subsection{Related work}

This work draws from the fields of optimization, complexity analysis and orthogonal polynomials, of which we comment on the most closely related ideas.

\paragraph{Average-case analysis.}
The average-case analysis has a long history in computer science and numerical analysis. Often it is used to justify the superior performances of algorithms such as Quicksort \citep{Hoare1962Quicksort} and the simplex method in linear programming~\citep{spielman2004smoothed}. 
Despite this rich history, it's challenging to transfer these ideas into continuous optimization due to 
the ill-defined notion of a typical continuous optimization problem.

In the context of optimization, \citet{pedregosa2020acceleration} derived a framework for analyzing the average-case gradient-based methods and developed methods that are non-asymptotic optimal algorithms with respect to the average-case. 
Such average-case analysis finds applications in various domains. For instance, \citet{lacotte2020optimal} use this framework to derive optimal average-case algorithms to minimize least-squares with random matrix sketching. Prior to this stream of papers, \citet{berthier2018accelerated} use methods based on Jacobi polynomials to design average-case optimal gossip methods, but without generalizing the framework.

In the numerical analysis literature,~\citet{deift2019conjugate} have recently developed an average-case complexity of conjugate gradient.

\paragraph{Asymptotics or orthonormal polynomials.} A key ingredient of the proof are asymptotics or orthonormal polynomials. This is a vast subject with applications in stochastic processes~\citep{grenander1958toeplitz}, random matrix theory~\citep{deift1999orthogonal} and numerical integration~\citep{mhaskar1997introduction} to name a few. The monograph of \citep{lubinsky2000asymptotics} discusses all results used in this paper.

\paragraph{Notation.}  Throughout the paper we denote vectors in lowercase boldface ($\xx$), matrices in uppercase boldface letters ($\boldsymbol H$), and polynomials in uppercase latin letter ($P, Q$). We will sometimes omit integration variable, with the understanding that $\int \varphi \dif \mu$ is a shorthand for $\int \varphi(\lambda) \dif\mu(\lambda)$.

\section{Average-Case Analysis of Gradient-Based Methods}\label{scs:average_analysis}

The goal of the average-case analysis is to quantify the expected error $\EE \|\xx_t - \xx^\star\|^2$, where $\xx_t$ is the $t$-th update of some optimization method and the expectation is taken over all possible problem instances. To make this analysis tractable, and following \citep{pedregosa2020acceleration}, we consider quadratic optimization problems of the form
\begin{empheq}[box=\mybluebox]{equation*}\tag{OPT}\label{eq:quad_optim}
  \vphantom{\sum_0^i}\min_{\xx \in \RR^d} \Big\{ f(\xx) \defas\!\mfrac{1}{2}(\xx\!-\!\xx^\star)^\top\!\HH(\xx\!-\!\xx^\star) \Big\},
\end{empheq}
where $\HH \in \RR^{d \times d}$ is a \textit{random} symmetric positive-definite matrix and $\xx^\star$ is a \textit{random} $d$-dimensional vector which is a solution of \eqref{eq:quad_optim}.

\begin{remark} Problem \eqref{eq:quad_optim} subsumes the quadratic minimization problem $\min_{\xx} \xx^\top\HH\xx + \bb^\top \xx + c$ but the notation above will be more convenient for our purposes.
\end{remark}

\begin{remark}
The expectation in $\EE \|\xx_t - \xx^\star\|^2$ is over the inputs and not over any randomness of the algorithm, as is common in the stochastic literature. In this paper we only consider deterministic algorithms.
\end{remark}

We consider in this paper the class of \textit{first order methods}, which build $\xx_t$ using a pre-defined linear combination of an initial guess and previous gradients:
\begin{equation}
    \xx_{t} \in \xx_0 + \textbf{span}\{\nabla f(\xx_0),\;\ldots,\;\nabla f(\xx_t)\}. \label{eq:first_order_definition}
\end{equation}
This wide class includes most gradient-based optimization methods, such as gradient descent and momentum. However, it excludes  quasi-Newton methods, preconditioned gradient descent or Adam (to cite a few), as the preconditioning allows the iterates to go outside span.

\subsection{Tools of the trade: orthogonal polynomials and spectral densities}

Average-case optimal methods rely on two key concepts that we now introduce: \textit{residual orthogonal polynomials} and the \textit{expected spectral distribution}.

\subsubsection{Orthogonal (residual) polynomials}

This section defines orthogonal polynomials and residual polynomials.

\begin{definition} \label{def:orthogonal_polynomial}
    Let $\alpha$ be a non-decreasing function such that $\int Q\dif\alpha$ is finite for all polynomials $Q$.
    We will say that the sequence of polynomials $P_0, P_1, \ldots$ is orthogonal with respect to $\dif\alpha$ if $P_i$ has degree $i$ and 
\begin{equation}\label{eq:optimal_orthogonal_polynomials}
    \int_{\mathbb{R}} P_i\,P_j\dif\alpha \begin{cases}
    = 0 & \text{if } i\neq j \\
    > 0 & \text{if } i = j
    \end{cases}.
\end{equation}
Furthermore, if they verify $P_i(0) = 1$ for all $i$, we call these {\bfseries residual orthogonal} polynomials.
\end{definition}

Residual orthogonal polynomials verify a three-term recurrence \citep[\S 2.4]{fischer1996polynomial}, that is, there exists a sequence of real values $a_t, b_t$ such that
\begin{equation}\label{eq:recurence_orthogonal_polynomials}
    P_{t}(\lambda)
    = (a_t + b_t \lambda) P_{t-1}(\lambda) + (1-a_{t})P_{t-2}(\lambda)\,,
\end{equation}
where $P_0(\lambda) = 1$ and $P_1(\lambda) = 1+b_1\lambda$\,.

\subsubsection{Expected spectral distribution}

The expected spectral distribution and the extreme eigenvalues of the matrix $\HH$ play similar roles in the case of, respectively, average-case and worst-case optimal methods. They measure the problem's difficulty and define the optimal method's parameters.
\begin{definition}[Empirical/Expected Spectral Measure]
Let $\HH$ be a random matrix with eigenvalues $\{\lambda_1, \ldots, \lambda_d\}$. The {\bfseries empirical spectral measure} of $\HH$, called ${\mu}_{\HH}$, is the probability measure
\begin{equation}\label{eq:empirical_spectral_density}
    \mu_{\HH}(\lambda) \defas {\textstyle{\frac{1}{d}\sum_{i=1}^d}} \delta_{\lambda_i}(\lambda) ~,
\end{equation}
where $\delta_{\lambda_i}$ is the Dirac delta, a distribution equal to zero everywhere except at $\lambda_i$ and whose integral over the entire real line is equal to one.

Since $\HH$ is random, the empirical spectral measure $\mu_{\HH}$ is a random measure. Its expectation over $\HH$,
\begin{equation}
\mu \defas \EE_{\HH}[\mu_{\HH}]\,,
\end{equation}
is called the {\bfseries expected spectral distribution}.
\end{definition}

\begin{example}[Marchenko-Pastur density and large least squares problems]
Consider a matrix $\AA \in \RR^{n \times d}$, where each entry is an iid random variable with mean zero and variance $\sigma^2$. Then it is known that the expected spectral distribution of $\HH = \frac{1}{n}\AA^\top\!\AA$ converges to to the Marchenko-Pastur distribution \citep{marchenko1967distribution} as $n$ and $d \to \infty$ at a rate in which the asymptotic ratio $d / n \rightarrow r$ is finite.
The Marchenko-Pastur distribution $\MP$ is defined as
\begin{equation} \label{eq:mp_distribution}
     \max\{1 - \tfrac{1}{r}, 0\}\delta_0(\lambda) + \frac{\sqrt{(L - \lambda)(\lambda - \ell)}}{2 \pi \sigma^2 r \lambda}1_{\lambda \in [\ell, L]}\,.
\end{equation}
Here $\ell\defas \sigma^2(1 - \sqrt{r})^2$, $L \defas \sigma^2(1 + \sqrt{r})^2$ are the extreme nonzero eigenvalues,  $\delta_0$ is a Dirac delta at zero (which disappears if $r \geq 1$) and $1_{\lambda \in [\ell, L]}$ is a rectangular window function, equal to 1 for $\lambda \in [\ell, L]$ and 0 elsewhere.

\end{example}

\subsection{Average-case optimal methods}

With these two ingredients, we can construct the method with optimal average-case complexity. We rewrite the expected error as an integral with weight function the expected spectral density $\mu$.

\begin{theorem}\label{thm:pedregosa_rate}\citep{pedregosa2020acceleration}
Assume $\xx_0$, $\xx^{\star}$ are random variables independent of $\HH$,  satisfying $\mathbb{E}[(\xx_0-\xx^\star)(\xx_0-\xx^\star)^\top] = R^2\boldsymbol{I}$. Let $\xx_t$ be generated by a first-order method, associated to the polynomial $P_t$. Then the expected error at iteration $t$ reads
\begin{empheq}{equation}\label{eq:error_norm_x}
  \vphantom{\sum_0^i}\mathbb{E}\|\xx_t-\xx^\star\|^2 = {\overbrace{R^2\vphantom{R_t}}^{\text{initialization}}}\int_\RR {\underbrace{P_t^2}_{\text{algorithm}}} {\overbrace{\dif\mu}^{\text{problem}}}
  \,.
\end{empheq}
\end{theorem}

The optimal first order method is obtained by minimizing the above identity over the space of residual polynomials of degree~$t$. This turns out to be equivalent to finding a sequence of residual polynomials $\{P_i\}$ orthogonal w.r.t. the weight function $\lambda \mu(\lambda)$, as shown in the following theorem.

\begin{theorem}\label{thm:optimal_polynomial}
\citep{pedregosa2020acceleration}
Let $a_t$ and $b_t$ be the coefficients of the three-term recurrence \eqref{eq:recurence_orthogonal_polynomials} for the sequence of residual polynomials orthogonal w.r.t. $\lambda \dif\mu(\lambda)$. 
Then the following method has optimal average-case complexity over the class of problems \eqref{eq:quad_optim}:\footnote{Throughout the paper, we will color-code {\color{colormomentum} momentum} and {\color{colorstepsize}step-size} parameters.}
\begin{align} \label{eq:optimal_algorithm}
    &\xx_1 = \xx_0 + {\color{colorstepsize}b_1} \nabla f(\xx_0),\\
    &\xx_{t} = \xx_{t-1} + {\color{colormomentum}(a_t-1)}(\xx_{t-1}-\xx_{t-2}) + {\color{colorstepsize}b_t} \nabla f(\xx_{t-1})\nonumber\,.
\end{align}
\end{theorem}

Due to the dependency of the coefficients $a_t, b_t$ on the expected spectral distribution, equation \eqref{eq:optimal_algorithm} does not represents a single scheme, but rather a family of algorithms: each different expected spectral distribution generates a different optimal method. Below is an example of such optimal algorithm w.r.t the Marchenko-Pastur expected spectral distribution.

\begin{example}[Marchenko-Pastur acceleration]\label{ex:mp}
Let $\dif\mu$ be the density associated with the Marchenko-Pastur distribution. Then, the recurrence of the optimal average-case method associated with this distribution is
\begin{equation*}\label{algo:mp_algo}
\begin{split}
    &\rho = \mfrac{1+r}{\sqrt{r}}\,,~\delta_{0} = 0;\\
    & \xx_1 = \xx_0-{\color{colorstepsize}\mfrac{1}{(1+r)\sigma^2}}\nabla f(\xx_0)\,;~\\
    &\delta_{t} = -({\rho+\delta_{t-1}})^{-1}\,; \\ 
    &\xx_{t} = \xx_{t-1} + {\color{colormomentum}\left(1 + \rho\delta_t\right)}(\xx_{t-2} - \xx_{t-1}) +  {\color{colorstepsize}\mfrac{ \delta_t}{\sigma^2\sqrt{r}}}\nabla f(\xx_{t-1})\,.
\end{split}
\end{equation*}
The coefficients come from the orthogonal polynomials w.r.t. $\lambda \dif\mu(\lambda)$, which is a shifted Chebyshev polynomials of the second kind.
\end{example}

\subsection{Polyak Momentum and worst-case optimality}

The Polyak momentum algorithm~\citep{polyak1964some}
has an optimal worst-case convergence rate over the class of first order methods with constant coefficients~\citep{polyak1987introduction, scieur2017integration}. 
The method requires knowledge of the smallest and largest eigenvalue of the Hessian $\HH$ (denoted $\ell$ and $L$ respectively) and iterates as follows:
\begin{align}\label{algo:pm_algo}\tag{PM}
&\xx_1 = \xx_0 - {\color{brown}\mfrac{2}{L + \ell}}\nabla f(\xx_0)\\
&\xx_{t+1} = \xx_t + {\color{colormomentum}\textstyle{\Big(\frac{\sqrt{L}-\sqrt{\ell}}{\sqrt{L}+\sqrt{\ell}}\Big)^2}}(\xx_{t} - \xx_{t-1}) - {\color{colorstepsize} \textstyle\Big(\frac{ 2}{\sqrt{L}+\sqrt{\ell}}\Big)^2}\nabla  \nonumber
    f(\xx_t)
\end{align}

\vspace{0.5em}\begin{remark} Unlike the Marchenko-Pastur accelerated method of Example~\ref{ex:mp}, coefficients of this method are constant in the iterations. Furthermore, these coefficients only depend on the edges  of the spectral distribution and not on the full density.
\end{remark}


\break

\section{All Roads Lead to Polyak Momentum}\label{scs:asymptotic_heavyball}


\begin{myenv}{Main result}
\begin{theorem}\label{thm:conv_hb}
    Assume the density function $\dif\mu$ is strictly positive in the interval $[\ell, L]$ with $\ell > 0$ and let ${a_t}, b_t$ be the parameters of the optimal average-case method (Theorem~\ref{thm:optimal_polynomial}). Then these parameters converge to those of \eqref{algo:pm_algo}. More precisely, we have the limits:
    \begin{align} \label{eq:limit_method}
        &\lim_{t \to \infty} {\color{colormomentum} a_t-1} = \underbrace{\left(\mfrac{\sqrt{L}-\sqrt{\ell}}{\sqrt{L}+\sqrt{\ell}}\right)^2}_{=\,\text{\eqref{algo:pm_algo} momentum}}\,, \;\;\quad\text{ and }\\
        &\lim_{t \to \infty}{\color{colorstepsize} b_t} = \underbrace{- \left(\mfrac{2}{\sqrt{L} + \sqrt{\ell}}\right)^2}_{=\,\text{\eqref{algo:pm_algo} step-size}}\,.
    \end{align}
~
\end{theorem}
\end{myenv}

The key insight of the proof is to cast the three-term recurrence of residual \textit{orthogonal} polynomials into orthonormal polynomials\footnote{A sequence $Q_1, Q_2, ...$ of orthogonal polynomials with respect to $\dif\omega$ is orthonormal if $\int Q_i^2 \dif \omega = 1$.}  in the interval $[-1,\,1]$. Once this is done, we will use asymptotic properties of these polynomials. The proof is split into three steps.
\begin{itemize}[leftmargin=*]
    \item Step 0 introduces notation and some known results.
    \item Step 1 writes the coefficients of optimal average-case methods in terms of properties of a class of orthonormal polynomials in the $[-1, 1]$ interval.
    \item Step 2 computes the limits of the expressions derived in the previous step by using known asymptotic properties of orthonormal polynomials.
\end{itemize}

{\bfseries \underline{Step 0}: Definitions.} In the classical theory of orthogonal polynomials, the weight function associated with orthogonal polynomials is defined in the interval $[-1, 1]$. However, in our case the spectral densities are instead defined in $[\ell, L]$. To translate results from one setting to the other we define the following linear mapping from $[\ell, L]$ to $[-1, 1]$:
\begin{equation}
    m(\lambda) = \mfrac{L+\ell}{L - \ell} - \mfrac{2}{L-\ell}\lambda\,.
\end{equation}
For notational convenience, we will also use the shorthand $m_0 \defas m(0)$.
We now define $Q_i(m(\cdot))$ as the $i$-th degree orthonormal polynomial with respect to the weight function $\lambda\mu(\lambda)$. That is, the sequence $Q_1, Q_2, \ldots$ verifies
\begin{equation}
    \int_{\ell}^L Q_i(m(\lambda))Q_j(m(\lambda)) \lambda \dif\mu(\lambda) = \begin{cases} 1 \text{ if $i=j$} \\ 0 \text{ otherwise\,.}\end{cases}\,,
\end{equation}
where $\delta_{ij}$ represents Kronecker's delta.
Like residual orthogonal polynomials, orthonormal  polynomials also verify a three-term recurrence relation. This time, the relation is of the form
\begin{equation}\label{eq:recurrence_orthonormal_poly}
    \alpha_t Q_t(\xi) = (\xi - \beta_t)Q_{t-1}(\xi) - \alpha_{t-1}Q_{t-2}(\xi)\,,
\end{equation}
and depends on coefficients $\alpha_t, \beta_t$:

{\bfseries \underline{Step 1}: Parameters of optimal method and orthonormal polynomials}.
In this step we derive the recurrence relation for an \emph{orthonormal} family with respect to the density $\dif\nu$. This will allow us to use existing results on the asymptotics of orthonormal polynomials.

\begin{lemma}\label{lemma:orthonormal_a_b}
Let $a_t, b_t$ be the parameters associated with optimal average-case method (Theorem~\ref{thm:optimal_polynomial}). These coefficients verify the following identity,
\begin{align}\label{eq:reformation_a_b}
    &{\color{colormomentum}1 - a_t} = -\frac{\alpha_{t-1}}{ \alpha_t}\frac{Q_{t-2}(m_0)}{ Q_t(m_0)} ~, \quad \text{ and }\\
    &{\color{colorstepsize} b_t} =  - \frac{2}{\alpha_t (L - \ell)} \frac{Q_{t-1}(m_0)}{Q_{t}(m_0)}~.
\end{align}
\end{lemma}
\begin{proof}
Since orthogonality is preserved after multiplication by a scalar, the polynomial $P_t(\lambda) \defas Q_t(m(\lambda)) / Q_t(m_0)$ is also orthogonal with respect to the weight function $\lambda \dif\mu(\lambda)$. The normalization  $1/Q_t(m_0)$ ensures $P_t$ is a residual polynomial. Note that $Q_t(m_0)$ cannot be zero because $m_0$ lies outside of the weight function's support $[-1,\,1]$.

Using Theorem~\ref{thm:optimal_polynomial}, the coefficients of the optimal average-case method can be derived from the three-term recurrence of this polynomial. Indeed, starting from the three-term recurrence of $Q_i$ \eqref{eq:recurrence_orthonormal_poly}, we obtain for $P_t$
\begin{align*}
P_t(\lambda) &= (m(\lambda)  - \beta_{t-1})\mfrac{Q_{t-1}(m(\lambda))}{ \alpha_t \, Q_t(m_0)} - \alpha_{t-1}\mfrac{Q_{t-2}(m(\lambda))}{ \alpha_t \, Q_t(m_0)}\\
& = \underbrace{\mfrac{1}{\alpha_t}\left(\mfrac{L+\ell}{L - \ell}   - \beta_{t-1} - \mfrac{2}{L-\ell}\lambda\right)\mfrac{Q_{t-1}(m_0)}{ \, Q_t(m_0) }}_{= (a_t + b_t\lambda)} P_{t-1}\\
&\qquad - \underbrace{\mfrac{\alpha_{t-1}}{ \alpha_t}\mfrac{Q_{t-2}(m_0)}{ Q_t(m_0)}}_{=-(1 - a_t)} P_{t-2}(\lambda)\,,
\end{align*}
where in the last line we used the definition of $m$ and the identity $P_{i}(\lambda) = Q_i(m(\lambda))/Q_i(m_0)$ for $i=t-1$ and $i=t-2$.
Finally, matching the coefficients of this recurrence with \eqref{eq:recurence_orthogonal_polynomials} yields the identity in the Lemma.
\end{proof}

{\bfseries \underline{Step 2}: Asymptotics of orthonormal polynomials.} This step uses known result on asymptotics of orthonormal polynomials to compute the limit $t \to \infty$ of expressions derived in the previous step.

We use the following theorem on the asymptotic ratio between two successive orthonormal polynomials.

\begin{theorem}[\citep{rakhmanov1983asymptotics};\footnote{The original version of this theorem was stated for monic orthogonal polynomials but is valid for polynomials with other normalizations like orthonormal, see for instance \citep{lubinsky2000asymptotics, denisov2004rakhmanov}. } Ratio Asymptotics] \label{thm:rakhmanov1983asymptotics}
    Let $\{Q_i\}$ be a sequence of orthonormal polynomials with respect to a weight function strictly positive in $]-1, 1[$, and zero elsewhere. Then  we have the following limit for the ratio of polynomials evaluated outside of the support,
    \begin{equation}\label{eq:rakhmanov}
         \lim_{t\rightarrow\infty} \frac{Q_t(\xi)}{Q_{t-1}(\xi)} = \xi + \sqrt{\xi^2-1} \quad \text{for} \;\; \xi > 1\,.
    \end{equation}
\end{theorem}

We can use this result to compute the limit of the ratio $Q_{t-1}(m_0)/Q_t(m_0)$, that appears in \eqref{eq:reformation_a_b}, as $m(0)> 1$ (and thus is not in the interval $[-1,\,1]$):
\begin{align}
    \lim_{t \to \infty} \mfrac{Q_{t-1}(m_0)}{Q_t(m_0)} &\stackrel{\eqref{eq:rakhmanov}}{=} \Big(\mfrac{L + \ell}{L - \ell} + \sqrt{\big(\mfrac{L + \ell}{L - \ell}\big)^2-1}\big)\Big)^{-1} \nonumber\\
    &= \mfrac{\sqrt{L}-\sqrt{\ell}}{\sqrt{L}+\sqrt{\ell}}\,. \label{eq:limit_ratio_q}
\end{align}

The other dependency of Eq. \eqref{eq:reformation_a_b} on the iteration $t$ is through the coefficients $\alpha_t, \beta_t$. To compute the limits of these we use the following known asymptotics:\footnote{It can be shown that the last two theorems are equivalent \citep[Theorem 13]{nevai1979orthogonal}. However, it will be more convenient for our purposes to present them as independent results.}

\begin{theorem}[\citet{mate1985asymptotics}; Limits of recurrence coefficients] Under the same assumptions as Theorem \ref{thm:rakhmanov1983asymptotics}, the limits of the coefficients $\alpha_t, \beta_t$ in the orthonormal three-terms recurrence (Eq. \ref{eq:recurrence_orthonormal_poly}) is
    \begin{equation}\label{eq:mate}
        \lim_{t\rightarrow\infty}\alpha_t = \mfrac{1}{2}~,\qquad \lim_{t\rightarrow\infty}\beta_t = 0 ~.
    \end{equation}
\end{theorem}

Using this last theorem together with \eqref{eq:limit_ratio_q}, we have
\begin{align*}
\lim_{t \to \infty} {\color{colormomentum}(1 - a_t)} &~~\stackrel{\eqref{eq:reformation_a_b}}{=} 
-\left(\lim_{t \to \infty} \mfrac{\alpha_{t-1}}{ \alpha_t}\right) \left( \lim_{t \to \infty}\mfrac{Q_{t-2}(m_0)}{ Q_t(m_0)}\right) \nonumber\\
&\!\stackrel{(\ref{eq:limit_ratio_q}, \ref{eq:mate})}{=} -\Big(\mfrac{\sqrt{L}-\sqrt{\ell}}{\sqrt{L}+\sqrt{\ell}}\Big)^2,
\end{align*}
which is the claimed limit.

To conclude the proof, we compute the same limit for the step-size ${\color{colorstepsize}b_t}$:
\begin{align}
    \lim_{t \to \infty} {\color{colorstepsize}b_t} &\stackrel{\eqref{eq:reformation_a_b}}{=}  - \mfrac{2}{ L - \ell} \left(\lim_{t \to \infty} \alpha_t^{-1}\right) \left(\lim_{t \to \infty} \mfrac{Q_{t-1}(m_0)}{Q_{t}(m_0)}\right) \\
    &\!\!\!\!\stackrel{(\ref{eq:limit_ratio_q}, \ref{eq:mate})}{=} - \left(\mfrac{2}{\sqrt{L} + \sqrt{\ell}}\right)^2\,\,,~\widetilde{r}_1 = r_1, \widetilde{r}_0 = r_0.
\end{align}

\section{Asymptotic Expected Convergence Rates} \label{sec:proof}

The previous section showed convergence of the method's parameters to PM, but said nothing about its rate of convergence. This section fills this gap by providing the asymptotic convergence of the expected convergence rate $\mathbb{E}\|\xx_t-\xx^\star\|^2$.
More precisely, we show that the expected convergence rate converges to the rate of convergence of Polyak, and that this convergence rate is \textit{independent} of the probability distribution.

\begin{restatable}{thm}{conv_hb_rate} \label{thm:conv_hb_rate}
    Under the same assumptions of Theorem \ref{thm:conv_hb}, the asymptotic expected rate of convergence of the optimal method converges to the worst-case rate of convergence,
    \begin{equation}\label{eq:asympt_rate}
         \limsup_{t\rightarrow \infty}\sqrt[t]{\mathbb{E}\left[\mfrac{\|\xx_t-\xx^\star\|^2}{\|\xx_0-\xx^\star\|^2}\right]} = \left(\mfrac{\sqrt{L}-\sqrt{\ell}}{\sqrt{L}+\sqrt{\ell}}\right)^2.
    \end{equation}
\end{restatable}

\begin{proof}
Let $P_t$ be the residual orthogonal polynomial w.r.t. $\lambda \dif\mu(\lambda)$. \citet[Theorem 3.1]{pedregosa2020acceleration} showed that the expected rate of convergence for average-case optimal methods admits the following simple form
\begin{equation}
    \mathbb{E}\|\xx_t-\xx^\star\|^2 =  R^{2}\int_{\mathbb{R}} P_t  \dif\mu~.
\end{equation}
This form is particularly convenient for us, as we can then use the the three-term recurrence to obtain a recurrence of this expression. Let $r_t = \int_{\mathbb{R}} P_t \dif\mu$. After using the recurrence over $P_t$,
\begin{align}
    r_t &= \int_{\mathbb{R}} (a_t+\lambda b_t)P_{t-1}(\lambda) + (1-a_t) P_{t-2}(\lambda) \dif\mu(\lambda)\nonumber\\
    &= a_t \underbrace{\int_{\mathbb{R}} P_{t-1}\dif\mu}_{=r_{t-1}} + (1-a_t) \underbrace{\int_{\mathbb{R}}P_{t-2} \dif\mu}_{=r_{t-2}}\,,
\end{align}
where in the last identity we have used the orthogonality between $P_t$ and $P_0(\lambda)=1$ w.r.t. $\lambda \dif\mu(\lambda)$. In all, we have that the convergence rate $r_t$ is described by the recurrence
\begin{equation}
\begin{split}
    r_t &= a_t r_{t-1} + (1-a_t) r_{t-2}\,,\\
    r_1 &= 1 + b_1 \int_{\mathbb{R}} \lambda \dif\mu\,,~ r_0 = 1 \,.
\end{split}
\end{equation}

A classical result, often referred to as the Poincar{\'e}-Perron theorem (see for example \citet[Theorem C]{pituk2002more} or \citep[Thm. 8.11]{elaydi2005introduction}  ), states that if $a_t$ has a finite limit --guaranteed by the previous theorem and which we denote $a_{\infty}$-- then the recurrence has a fundamental set of solutions $\{r_t^1, r_t^2\}$ such that
\begin{equation}
    \limsup_{t \to \infty} \sqrt[t]{r^i_t} = |\lambda_i|\quad \text{ $i =1, 2$}\,,
\end{equation}
where $\lambda_i$ are the roots of the characteristic equation ${\lambda^2 - a_{\infty} \lambda - (1 - a_{\infty})}$. In our case, these roots are $1$ and $1 - a_{\infty}$. Now, since the method we're considering is average-case optimal, this limit cannot be larger than that of Polyak momentum, known to be $(\tfrac{\sqrt{L}-\sqrt{\ell}}{\sqrt{L}+\sqrt{\ell}})^2 < 1$. Hence, we can eliminate the solution $r^1_t = 1$ and conclude
\begin{equation}
    \limsup_{t \to \infty} \sqrt[t]{r_t} = (1 - a_{\infty}) = \left(\mfrac{\sqrt{L}-\sqrt{\ell}}{\sqrt{L}+\sqrt{\ell}}\right)^2\,.
\end{equation}

\end{proof}

\section{Discussion and Simulations: Speed of Convergence to PM}

The main result (Theorem~\ref{thm:conv_hb}) shows that, asymptotically, any average-case optimal method converge towards Polyak momentum. This could be interpreted as evidence against average-case optimal methods, as average-case optimal methods are not ``essentially different'' from PM. However, simulations show other dynamics at play.

\begin{figure*}
    \centering
    \includegraphics[width=\linewidth]{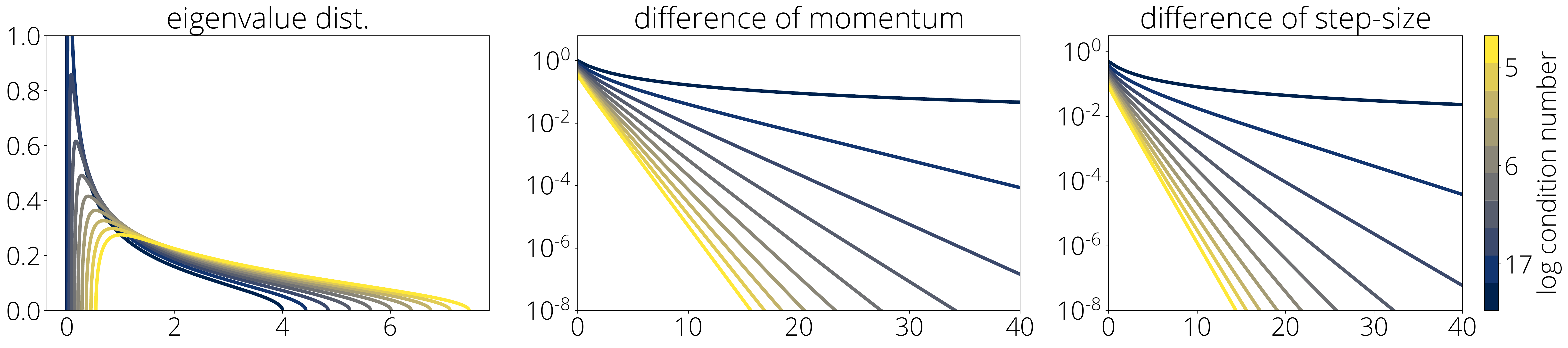}
    \includegraphics[width=\linewidth]{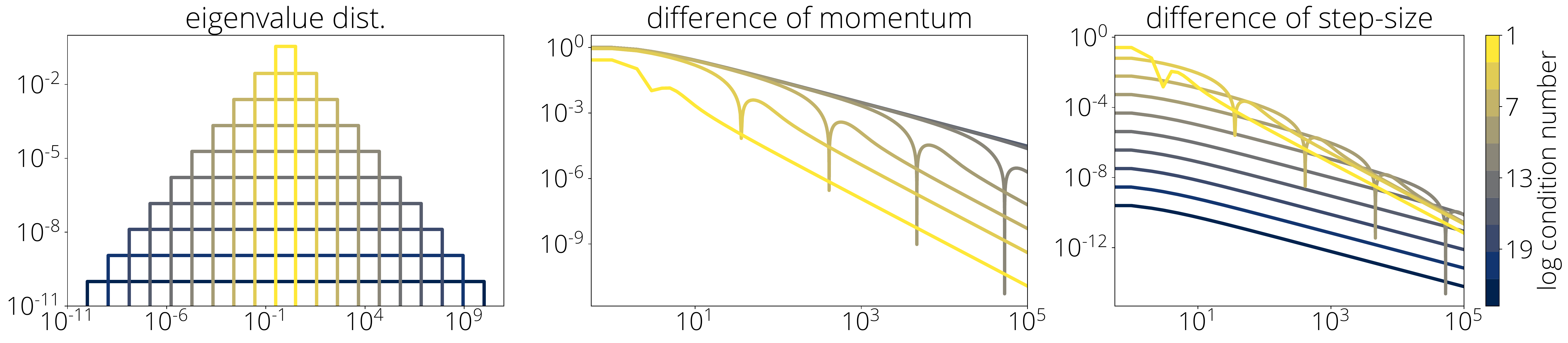}
    \caption{{\bfseries Speed of convergence to Polyak Momentum}. For different parametrizations of the Marchenko-Pastur (\textit{top line}) and uniform (\textit{bottom line}) distributions, we plot the absolute difference between the average-case optimal momentum parameter (\textit{middle}) and average-case optimal step-size (\textit{right}) and the momentum and step-size of the Polyak method.
    The plots show a high anti-correlation between the speed of convergence of optimal average-case methods to PM and problem conditioning: for well-conditioned problems (small condition number) the parameters converge faster to PM than for ill-conditioned (large condition number) problems.
    Thus, in a regime were we perform only a few iterations, Polyak momentum may not be the best choice.}
    \label{fig:speed_convergence}
\end{figure*}

In Figure \ref{fig:speed_convergence} we plot the speed of convergence of the parameters of the optimal average-method for the Marchenko-Pastur distribution with different ratios $r = \frac{d}{n}$ (and hence condition number) and for the uniform distribution with different intervals. We see a clear effect of the condition number on the speed of convergence. The more ill-conditioned the problem, the slower the convergence of the optimal method to PM, implying that PM behaves sub-optimally for a larger number of iterations. This observation is consistent with the results of \citep{pedregosa2020acceleration}, who showed important speedups in the ill-conditioned regime.

\section{Conclusion and Perspectives}

In this work, we've shown that optimal average-case methods for minimizing quadratics converge to PM under mild assumptions on the expected spectral distribution. This universality over the probability measure is somewhat surprising, as Polyak momentum method only depends on the edges of the spectrum, while on the other hand optimal average-case methods depend on the whole spectrum.

A potential area for future work is the analysis of the rate of convergence of optimal method to Polyak momentum algorithm. It seems the convergence of the step-size and momentum parameters are bounded polynomially in the number of iterations. This observation indicates the potential benefit of optimal methods over PM in the case where we perform a small number of iteration, typical in machine-learning problems.


A second research direction is the study of optimal polynomials on the \textit{complex} plane. In this case, we are no longer solving the optimization problem \eqref{eq:quad_optim}. Instead, we aim to solve the linear system $\AA \xx=\bb$, where the matrix $\AA$ is non-symmetric, with potentially complex eigenvalues. This has implication in the study of optimal algorithm in game theory \citep{azizian2020accelerating} or in the acceleration of primal-dual algorithms \citep{bollapragada2018nonlinear}.

Finally, our results are only valid in the strongly convex regime ($\ell > 0$), ruling out the important case $r=1$ in the Marchenko-Pastur distribution, which corresponds to large least squares problems with a square matrix. After the first version of this paper appeared, \citet{paquette2020} derived an average-case analysis for gradient descent and showed a gap between the asymptotic average-case and worst-case convergence rate. The development of average-case optimal methods and the study of their asymptotic limits in this regime remains an open problem.


\clearpage
\section*{Acknowledgements}

We would like to thank our colleague Gauthier Gidel for identifying and reporting some gaps in the proof of Theorem \ref{thm:conv_hb_rate}. A note of gratitude also goes to Reza Babanezad, Simon Lacoste-Julien, Remi Lepriol, Nicolas  Loizou, Adam Ibrahim, Nicolas Leroux and Courtney Paquette for their insightful discussions and relevant remarks. We also thank Francis Bach and Raphaël Berthier for their useful remarks and pointers.

\bibliography{biblio}
\bibliographystyle{icml2020}

\end{document}